\documentclass[11pt]{amsart}
\usepackage{geometry}                % See geometry.pdf to learn the layout options. There are lots.
\usepackage{amsthm}
\usepackage{graphicx}
\usepackage{amssymb}
\usepackage{epstopdf}
\usepackage{mathtools}
\newtheorem{theorem}{Theorem}[section]

\newtheorem{proposition}[theorem]{Proposition}
\newtheorem{corollary}[theorem]{Corollary}

\title{Obstructions to Shake Sliceness for Links}
\author{Anthony Bosman}
\date{}
\address{Department of Mathematics, Andrews University, 4260 Administration Dr., Berrien Springs, MI 49104}
\email{bosman@andrews.edu}

\begin{document}
\maketitle

\begin{abstract}
    Shake slice generalizes the notion of a slice link, naturally extending the notion of shake slice knots to links. There is also a relative version, shake concordance, that generalizes link concordance. We show that if two links are shake concordant, then their zero surgery manifolds are homology cobordant. Then we give several obstructions to a link being shake slice; for instance, the Arf invariants vanish for both the link and each component. Finally we show that a shake slice link bounds disjoint disks in a homology 4-ball and hence each component is algebraically slice.
\end{abstract}

\section{Introduction}

    It is well-known that concordance of links implies their associated zero surgery manifolds are homology cobordant. Of interest is when the zero surgery manifolds are homology cobordant but the links are are not concordant. For instance, Cochran, Franklin, Hedden, and Horn \cite{cochran13} exhibit non-concordant, topologically slice knots with homology cobordant zero surgery manifolds. Moreover, Cha and Powell \cite{chapowell14} provide an infinite family of links with unknotted components that all have identical Milnor invariants and homeomorphic zero surgery manifolds with homotopy class of meridians preserved, but none of which are pairwise concordant.
    
    The notion of concordance has been extended to the more general notion of shake concordance for knots \cite{cochranray16} and links \cite{bosman20}. For links, there is also a more restricted notion, still more general than concordance, called strong shake concordance. We show that shake concordance of links (and hence also strong shake concordance of links) implies homology cobordism.
    
    \begin{proposition}
     Suppose $m$-component links $L$ and $L'$ are shake concordant. Then the zero surgery manifolds $M_L$ and $M_L'$ are homology cobordant.
     \label{homology_cobordant}
     \end{proposition}
     
      Families of links that are shake concordant but not concordant offer further examples of non-concordant links with homology cobordant zero surgery manifolds.
    
    It follows from the homology cobordism that if individual link components are shake concordant to each other, then they share the same algebraic concordance class, but this only holds in the case of strong shake concordance. It turns out that nothing can be said about the relationship of individual components in the more general setting of shake concordance of links, for given any two knots $K$ and $J$, there exist 2-component shake concordant links that have $K$ and $J$ as their respective first component and an unknot as their second component as in Figure \ref{fig:knots_with_meridians} (see Proposition 3.1 of \cite{bosman20}).
    
    \begin{figure}
        \centering
        \includegraphics[height=1.2in]{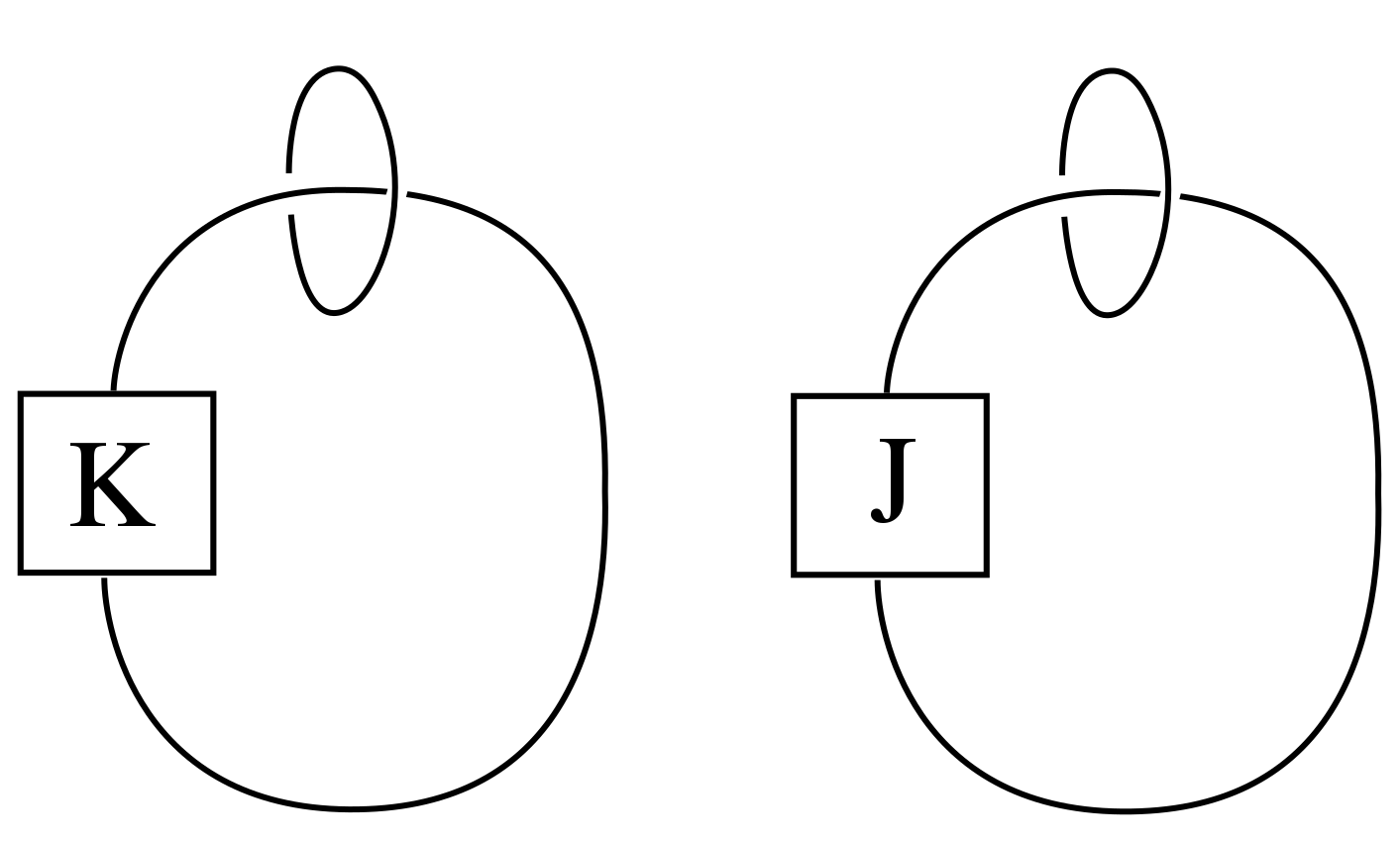}
        \caption{Shake concordant links}
        \label{fig:knots_with_meridians}
    \end{figure}
    
     Nevertheless, we can give a number of restrictions to a link being shake concordant to the trivial link, also called shake slice. For instance, we show:
    
    \begin{theorem}
		If a link $L=L_1\sqcup L_2\sqcup ...\sqcup L_m$  is shake slice, then:
	\begin{itemize}
		\item $Arf(L)=0$,
		\item $Arf(L_i)=0$ for $i=1,...,m$, and
		\item $Arf(L_i \sqcup L_j)=0$ for $i\neq j$.
	\end{itemize}
	\end{theorem}
	
	In fact, we show that if a link is shake slice, then it is homologically slice:
    
    \begin{proposition}
    \label{homologically_slice}
    Suppose the $m$-component link $L$ is shake slice. Then $L$ bounds $m$ disjoint disks in a homology 4-ball. That is, $L$ is homologically slice.
    \end{proposition}
    
    When Akbulut introduced the notion of shake slice knots in \cite{akbulut77}, he raised the still open question of if there exist slice knots that are not shake slice. The above proposition highlights the difficulty of resolving the generalized question for links as most invariants are known to fail to distinguish between slice and homologically slice.

\section{Definitions and Background}
    A $2n+1$ $r$-shaking of a knot $K$ is formed by taking $2n+1$ $r$-framed parallel copies of $K$ such that $n+1$ have the same orientation as $K$ and $n$ have opposite orientation. A knot $K\subset S^3=\partial B^4$ is called $r$-shake slice if we can form a 4-manifold $W_K^r$ by attaching a 2-handle, $D^2\times D^2$, to $B^4$ along $K$ with framing $r$ such that there exists an embedded sphere representing the generator of $H_2(W_K^r)\cong \mathbb{Z}$. Equivalently, $K\subset \partial B^4=S^3$ is $r$-shake slice if there exists a $2n+1$ $r$-shaking of $K$ that bounds a genus 0, connected, smooth surface in $B^4$ for some $n\geq0$. Note all slice knots are $r$-shake slice for any $r$. For $r\neq0$ there exists $r$-shake slice knots that are not slice (see \cite{akbulut77}, \cite{lickorish79}, \cite{cochranray16}); this remains open for $r=0$.
    
    We will restrict our attention to framing $r=0$.
    
    In \cite{bosman20} we extended the notion of shake slice to links: given an $m$-component link $L$, we can form a 4-manifold $W_L$ by attaching $m$ 2-handles to $B^4$ along the components of $L$ with framing $0$. We then say the link $L$ is shake slice if there exists embedded spheres representing the generators $(1,0,...,0), ...,(0,0,...,1)$ of $H_2(W_L)\cong \mathbb{Z}^m$. We call the link strongly shake slice if each sphere representing the $i^{th}$ generator only intersects the $i^{th}$ 2-handle, for $i=1,...,m$. See Figure \ref{fig:shake_slice_spheres}.
    
    \begin{figure}[!tbp]
        \centering
        \begin{minipage}[b]{0.45\textwidth}
        \includegraphics[width=\textwidth]{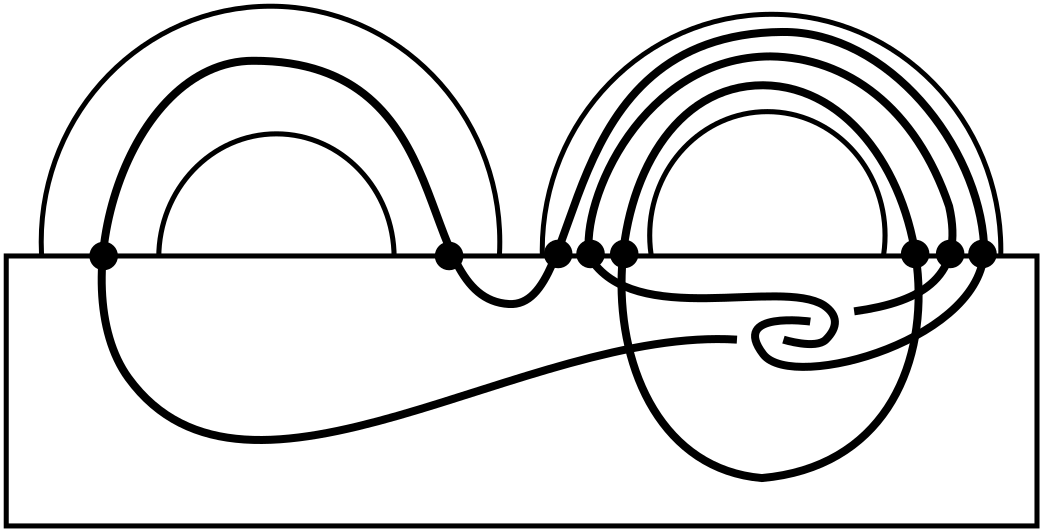}
        
        \end{minipage}
        \hfill
        \begin{minipage}[b]{0.45\textwidth}
        \includegraphics[width=\textwidth]{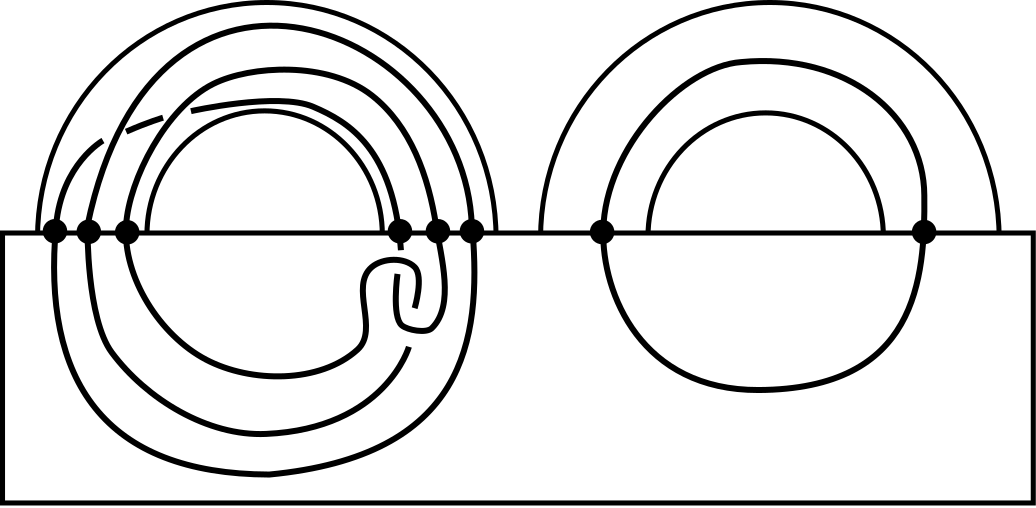}
        \end{minipage}
        \label{fig:shake_slice_spheres}
        \caption{A shake slice (left) and strong shake slice (right) link.}
    \end{figure}
    
    We can also give a geometric formulation: a $(2n_1+1,...,2n_m+1)$ shaking of a $m$-component link $L$ is formed by taking a $2n_k+1$ shaking of each component $L_k$ of $L$. Then we say a link $L$ is shake slice if there exists disjoint, smooth, properly embedded, compact, genus zero surfaces $\Sigma_1,...,\Sigma_m$ in $B^4$ that bound a $(2n_1+1,...,2n_m+1)$-shaking of $L$ as follows:
    \begin{itemize}
        \item each $\Sigma_i$ bounds $2n_{ii}+1$ 0-framed parallel copies of $L_i$ (exactly $n_{ii}+1$ of which the same orientation as $L_i$),
        \item $2n_{ij}$ 0-framed parallel copies of $L_j$ for each $j\neq i$ (exactly $n_{ij}$ of which the same orientation as $L_j$),
        \item $\sum_{j=1}^m n_{ij}=n_i$ for each $i=1,...,m$.
    \end{itemize}
    
    A link is then strongly shake slice exactly when $n_{ij}=0$ for all $i\neq j$.
    
    Slice links are shake slice, for the slice disks together with the core of the 2-handles form the desired spheres.
    
    There is also a relative version: We say $m$-component links $L$ and $L'$ are $(2n_1+1,...,2n_m+1;2n'_1+1,...,2n'_m+1)$ shake concordant if there exists $m$ disjoint, smooth, properly embedded, compact, connected, genus zero surfaces $F_1$, ..., $F_m$ in $S^3\times[0,1]$ such that
    \begin{itemize}
        \item $F_k\cap(S^3\times\{0\})$ consists of a $2n_{ii}+1$ shaking of $L_i$ and $n_{ij}$ pairs of oppositely oriented framed parallel copies of $L_j$ for each $j\neq i$ such that $\sum_k n_{ik}=n_i$.
        \item $F_k\cap(S^3\times\{1\})$ consists of a $2n'_{ii}+1$ shaking of $L'_i$ and $n'_{ij}$ pairs of oppositely oriented framed parallel copies of $L'_j$ for each $j\neq i$ such that $\sum_k n'_{ik}=n'_i$.
    \end{itemize}
    for all $1\leq k\leq m$. If, in addition, $n_{ij}=n'_{ij}=0$ for all $i\neq j$, then we call the links $L$ and $L'$ strongly shake concordant.
    
    \begin{figure}[!tbp]
        \centering
        \begin{minipage}[b]{0.45\textwidth}
        \includegraphics[width=\textwidth]{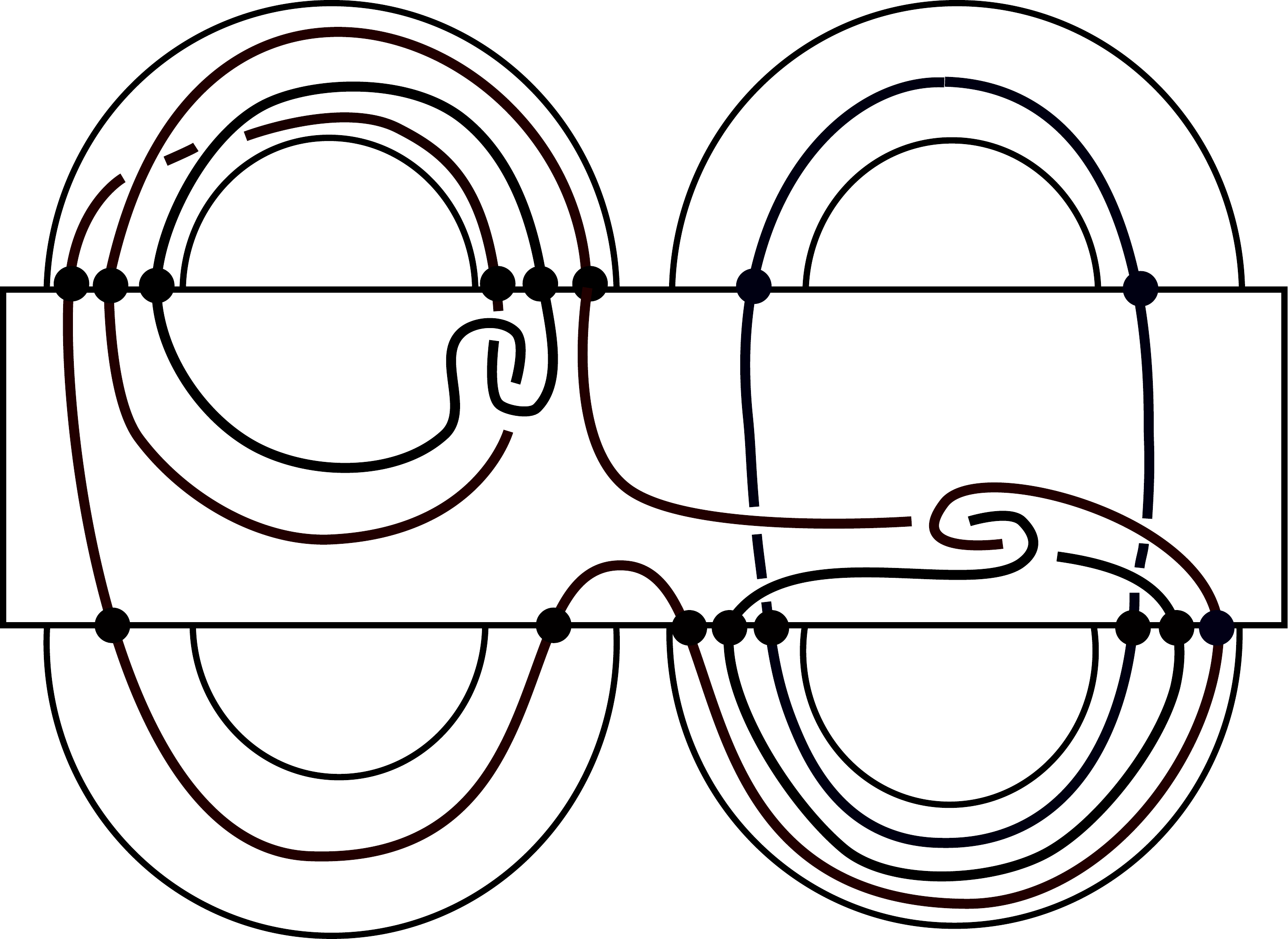}
        \end{minipage}
        \hfill
        \begin{minipage}[b]{0.45\textwidth}
        \includegraphics[width=\textwidth]{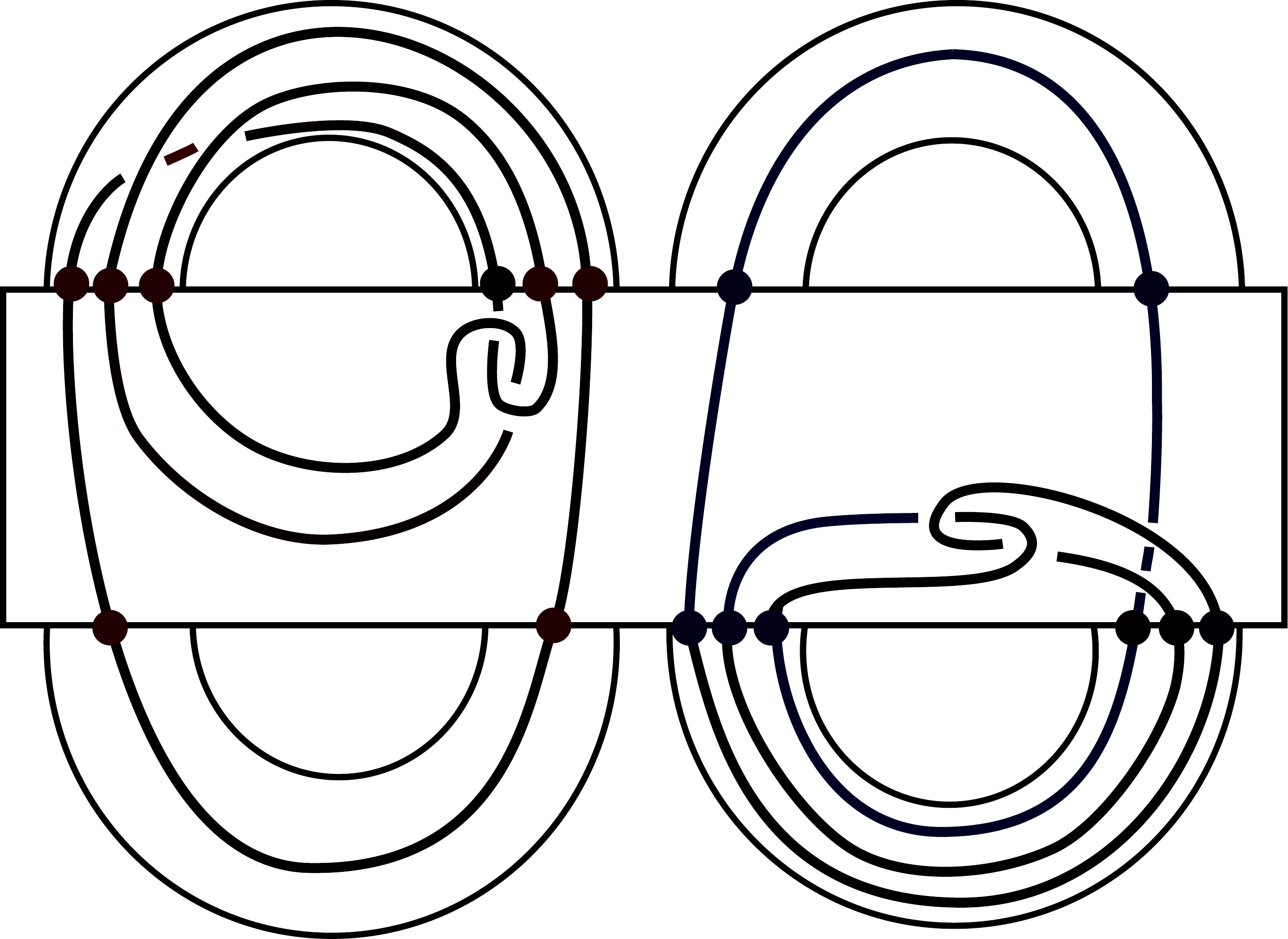}
        \end{minipage}
        \label{fig:sh_conc}
        \caption{A shake concordance (left) and strong shake concordance (right) between 2-component links.}
    \end{figure}
    
    Milnor introduced the $\bar{\mu}_L(I)$ invariants in \cite{milnor54} and \cite{milnor57}; they are a family of invariants for a given $m$-component link $L$ defined for multi-index $I=i_1i_2\cdots i_k$ where $1\leq i_j\leq m$. We let $|I|=k$ denote the length of the multi-index. They are defined algebraically from the link group by measuring how deep longitudes lie in the lower central series of the group, which gives rise to an indeterminacy in the higher order invariants if the lower invariants do not vanish. Where well defined, the $\bar{\mu}$ invariants are concordance invariants. To avoid the indeterminacy, study is often focused on the first non-vanishing Milnor invariants; that is, non-zero $\bar{\mu}_L(I)$ such that $\bar{\mu}_L(I')=0$ for all $|I'|<|I|$. Note that the $\bar{\mu}$ invariants generalize linking number; for instance, $\bar{\mu}_L(ij)=lk(L_i,L_j)$ and those of greater length, such as $\bar{\mu}(123)$, are able to detect higher order linking as in the Borromean rings. For a slice link, the $\bar{\mu}$ invariants all vanish; though the converse is not true.
    
\section{Examples of Shake Concordant and Slice Links}
    Given any two knots $K$ and $J$, we can add an unknotted meridianal component to each to obtain shake concordant links as in Figure \ref{fig:knots_with_meridians}. Hence, invariants of components of a link are not shake concordance invariants. Analogously, let $W(K)$ denote the 2-component link formed by the untwisted Whitehead double of $K$ and an unknotted component as in Figure \ref{fig:shaking_whitehead_double}. Note when the knot is the unknot $K$, $W(U)$ is the Whitehead link.
    
    \begin{proposition}
    For any two knots $K$ and $J$, the links $W(K)$ and $W(J)$ are shake concordant.
    \end{proposition}
    \begin{proof}
    Notice that by taking a 3-shaking of the unknotted component of $W(K)$ we are able to achieve a band pass move as in Figure \ref{fig:shaking_double_crossing} by fusing each arc of the band with copy of the unknot with same orientation. Thus, if we let $c_K$ denote the unknotting number of $K$, then by taking a $(2c_K+1)$-shaking of the unknotted component, we are able to completely unknot $K$ in $W(K)$ by accomplishing $c_K$ band pass moves. As we started with the untwisted Whitehead double, there will be no half twists in the resulting Whitehead link $W(U)$, as in Figure \ref{fig:shaking_whitehead_double}. That is, there is a $(1,2c_K+1;1,1)$-shake concordance between $W(K)$ and the Whitehead link $W(U)$. Similarly, there is a $(1,2c_J+1;1,1)$-shake concordance between $W(J)$ and the Whitehead link $W(U)$. Gluing the surfaces representing these shake concordances together gives a $(1,2c_K+1;1,2c_J+1)$-shake concordance between $W(K)$ and $W(J)$.
    \end{proof}
    \begin{figure}
        \centering
        \includegraphics[height=1.3in]{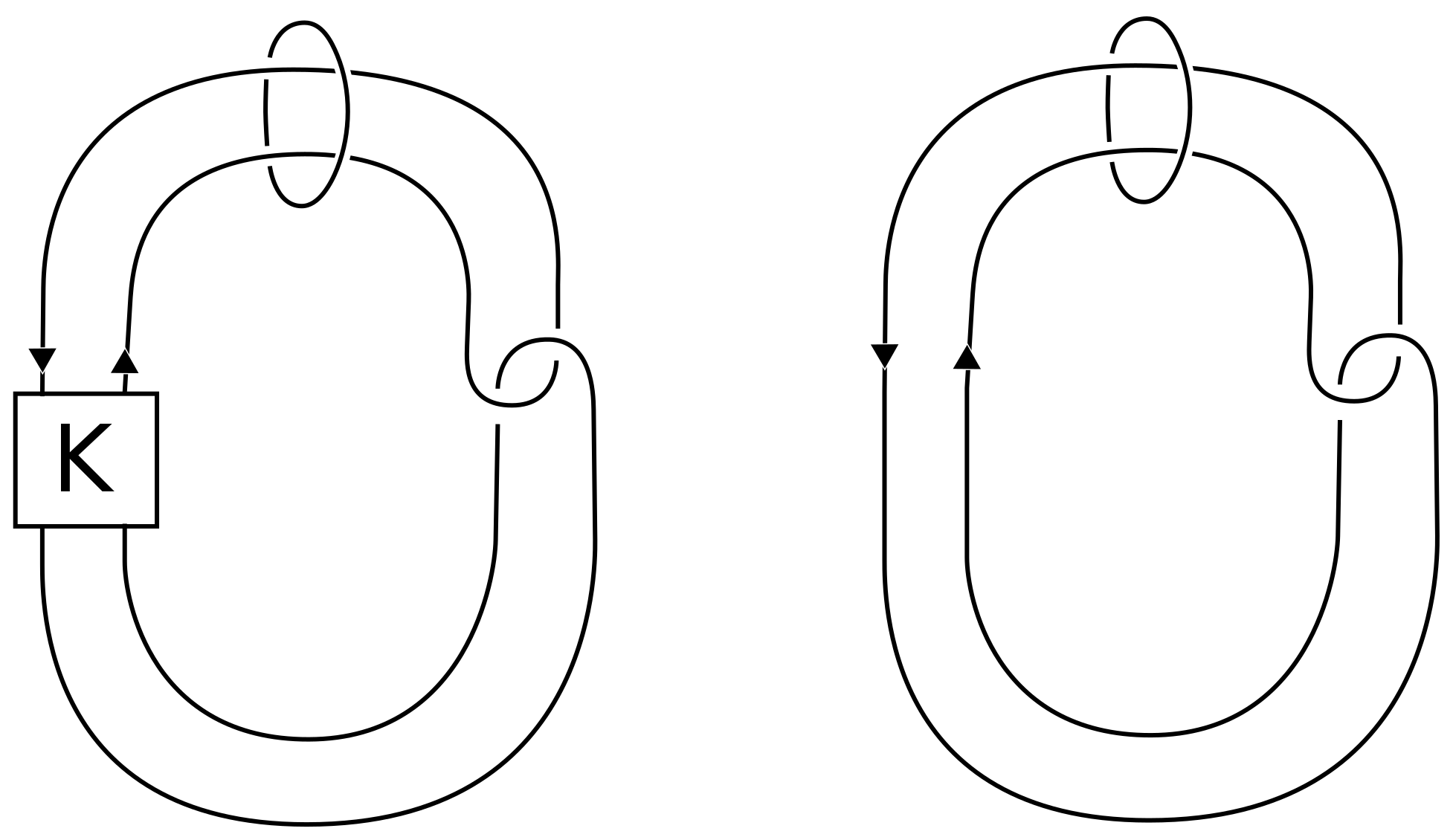}
        \caption{The link formed by the 0-framed Whitehead double of a knot $K$ and an unknotted component (left) is shake concordant to the Whitehead link (right).}
        \label{fig:shaking_whitehead_double}
    \end{figure}
    \begin{figure}
        \centering
        \includegraphics[height=1.3in]{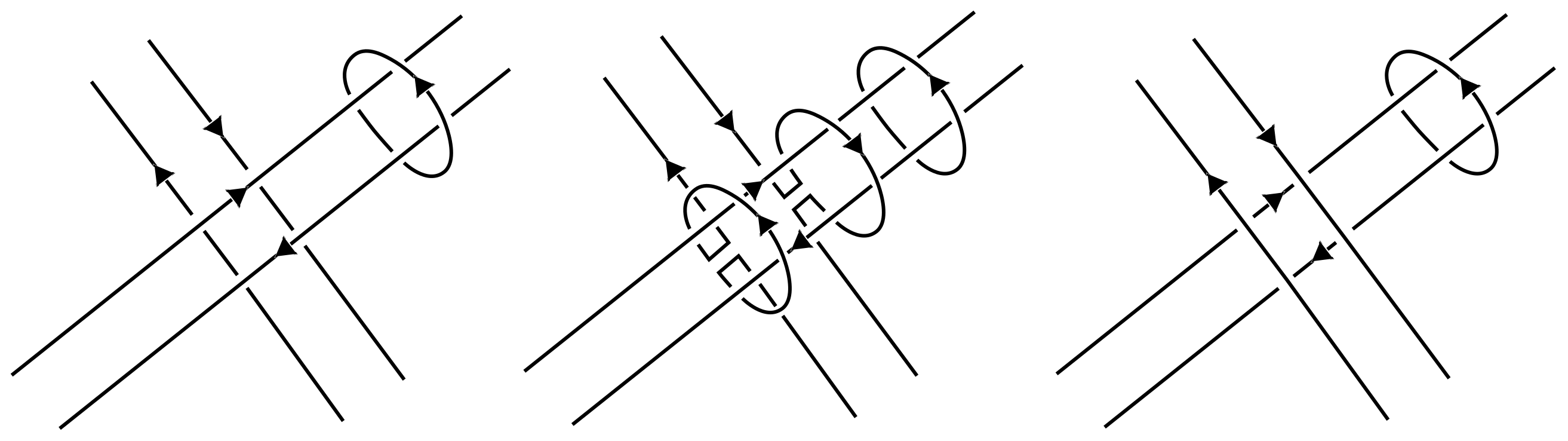}
        \caption{}
        \label{fig:shaking_double_crossing}
    \end{figure}
    
    Given a knot $K$ we can form the $2$-component Bing double of $K$ denoted $BD(K)$. It is well-known that the Bing double of a slice knot is a slice link; the following proposition generalizes this.
    \begin{proposition}
        Given a shake slice knot $K$, the Bing double of $K$, denoted $BD(K)$ is a strongly shake slice link.
    \end{proposition}
    \begin{proof}
        Since $K$ is shake slice, there exists a $2n+1$ shaking of $K\subset S^3$ that bounds a 0-genus surface in $B^4$. Then $4n+2$ parallels of $K$, with $2n+1$ of each orientation, will bounds two disjoint, smooth, 0-genus surfaces in $B^4$. We can attach pairs of oppositely oriented parallels of $K$ with bands passing through an unknotted component as in Figure \ref{fig:shaking_bing_double}. The bands fuse the surfaces into a single 0-genus surface $\Sigma$ in $B^4$. And we can push into $B^4$ the intersections of the bands with the disk $D$ bounded by the unknotted component to make $\Sigma$ and $D$ disjoint, smooth, $0$-genus surfaces in $B^4$. Notice the resulting link is a $(2n+1,1)$-shaking of $BD(K)$, the Bing double of $K$. Hence $BD(K)$ is strongly shake slice.
    \end{proof}
    \begin{figure}
        \centering
        \includegraphics[height=2.2in]{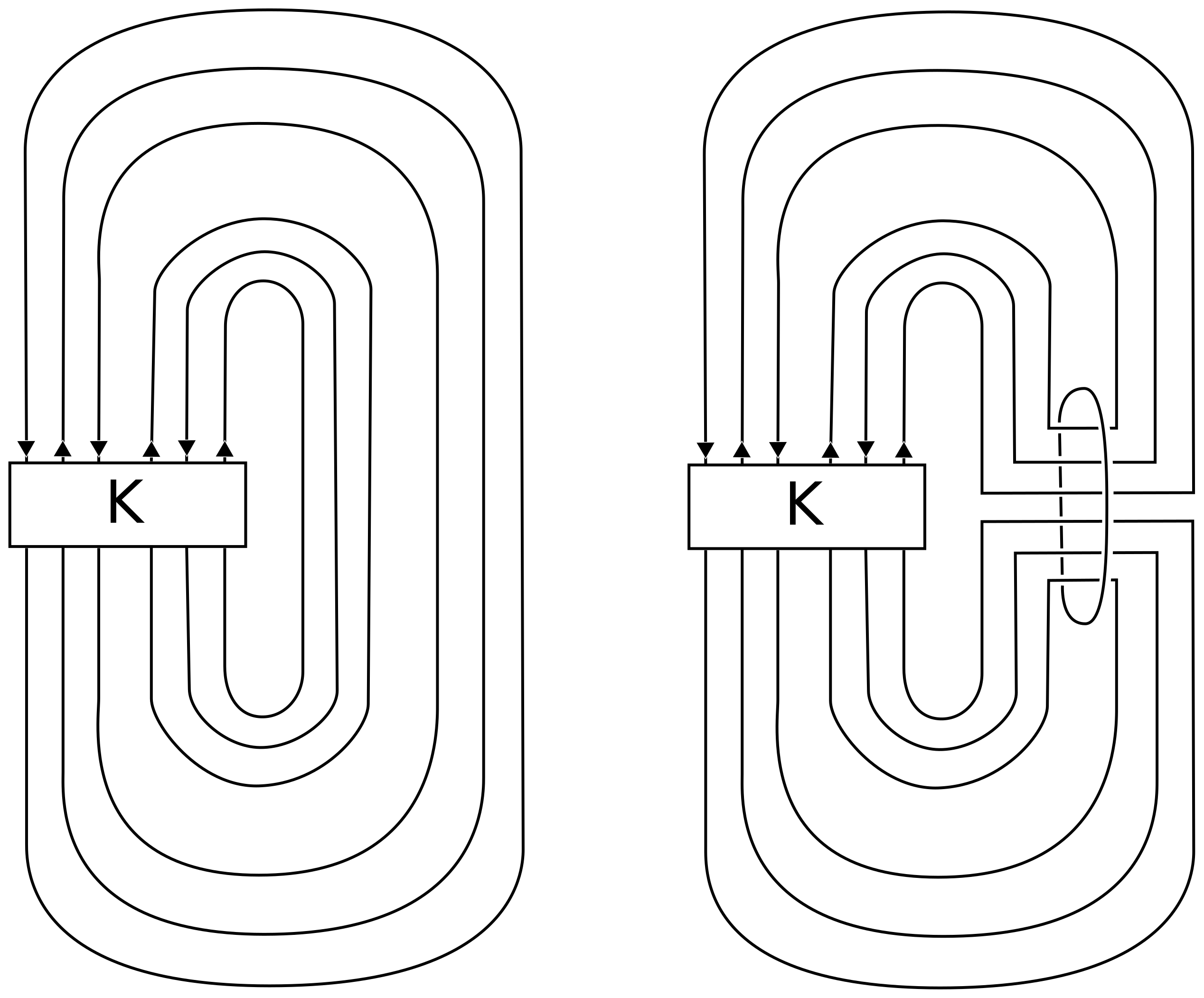}
        \caption{Attaching 3 bands to 6 parallel copies of $K$ (left) and adding an unknotted component to form a (3,1)-shaking of $BD(K)$ (right).}
        \label{fig:shaking_bing_double}
    \end{figure}

\section{Homology Cobordant}
    We call two closed, oriented 3-manifold $M_1$ and $M_2$ {\it homology cobordant} if there exists a compact, oriented 4-manifold $W$ such that $\partial W=M_1\sqcup -M_2$ and the maps induced by inclusion $H_n(M_i;\mathbb{Z})\rightarrow H_n(W;\mathbb{Z})$, $i=1,2$, are isomorphisms for all $n$.
 
     \begin{proposition}
     Suppose $m$-component links $L$ and $L'$ are shake concordant. Then the zero surgery manifolds $M_L$ and $M_L'$ are homology cobordant.
     \label{homology_cobordant}
     \end{proposition}
 
     \begin{proof}
     This has already been shown when $m=1$ in (\cite{cochranray16}, Proposition 5.1) which generalizes as follows. Recall $W_{L,L'}$, the 4-manifold obtained by attaching 2-handles along $L$ and $L'$, has boundary components $M_L$ and $-M_{L'}$. We have
     \[	H_n(W_{L,L'})\cong
     	\begin{cases}
    		\mathbb{Z} & n=0,3\\
    		0 & n=1\\
    		\mathbb{Z}^{2m} & n=2\\
    		0 & n\geq 4
    	\end{cases},\;\;\;\;\; H_n(M_L)\cong H_n(M_{L'})\cong
     	\begin{cases}
    		\mathbb{Z} & n=0, 3\\
    		\mathbb{Z}^{m} & n=1\\
    		\mathbb{Z}^{m} & n=2\\
    		0 & n\geq 4
    	\end{cases}.
    \]
 
    We will modify $W_{L,L'}$ such that the inclusion maps from $M_L$ and $M_{L'}$ into the modified 4-manifold induce isomorphisms on homology. Let $\Sigma_1,...,\Sigma_m\hookrightarrow W_{L,L'}$ be the embedded spheres guaranteed by the definition of shake concordance of links. We can perform surgery on each $\Sigma_i$ by removing a neighborhood of $\Sigma_i$, which is diffeomorphic to $S^2\times D^2$, and gluing in a copy of $D^3\times S^1$, which we can do since $\partial(S^2\times D^2)=S^2\times S^1=\partial(D^3\times S^1)$. Denote the resulting 4-manifold $\overline{W}$. Notice, this the effect of killing half the generators of the second homology group by killing $(\bar{e}_i,\bar{e}_i)$, for $i=1,...,m$. Also, this introduces $m$ generators for the first homology group. A Mayer-Vietoris argument verifies 
    \[	H_n(\overline{W})\cong
     	\begin{cases}
    		\mathbb{Z} & n=0, 3\\
    		\mathbb{Z}^{m} & n=1\\
    		\mathbb{Z}^{m} & n=2\\
    		0 & n\geq 4
    	\end{cases}
    \]
    and that the induced maps from inclusion give the desired isomorphisms.
    \end{proof}

     In \cite{harvey08} Harvey introduced the real-valued homology cobordism invariants $\rho_n$ for closed 3-manifolds. It follows from the above proposition that $\rho_n$, and other homology cobordism invariants, can be treated be invariants of shake concordance of links.

    Moreover, Levine's algebraic knot concordance class \cite{levine69a}, \cite{levine69b} is determined by the zero surgery manifold of a knot via the Blanchfield form and preserved under homology cobordism \cite{trotter73}. This give rise to the following corollary of Proposition \ref{homology_cobordant}:

    \begin{corollary}[Corollary 5.2 in \cite{cochranray16}]
    	If knots $K$ and $K'$ are shake concordant, then the algebraic concordance class of $K$ and $K'$ agree and hence $K$ and $K'$ have equal signatures and Arf invariants.
    \end{corollary}

    What can be said in the case of links? If $L$ and $L'$ are strongly shake concordant, then each corresponding pair of components $L_i$ and $L_i'$ are shake concordant. Therefore we conclude:

    \begin{corollary}
    If $L=L_1\sqcup ...\sqcup L_m$ and $L_1'\sqcup ... \sqcup L_m'$ are strongly shake concordant, then $L_i$ and $L_i'$ have the same algebraic concordance class and hence equal signatures and Arf invariant for all $i=1,...,m$.
    \end{corollary}

    If $L$ and $L'$ are shake concordant, but not strongly shake concordant, no knot invariant of concordance is preserved in the components of a link. Nevertheless, we can still find numerous obstructions to a knot being shake slice, as we'll see in the next section.

    \section{Obstructions to Shake Sliceness}
    Recall that the 4-genus of a link $L$ is defined by $g_4(L)=\min\{\sum_{i=1}^mg\left(\Sigma_i\right)\}$ where the $\Sigma_1\sqcup\Sigma_2\sqcup\cdots\sqcup\Sigma_m$ denote a collection of disjoint, smooth surfaces embedded in $B^4$ such that each $\Sigma_i$ bounds $L_i\subset \partial B^4$. We have the following bound on the 4-genus of a shake shake link.
    
    \begin{proposition}
    If $L$ is a $(2n_1+1,2n_2+1,...,2n_m+1)$ shake slice link, then $g_4(L)\leq \sum_{k=1}^mn_k$.
    \end{proposition}
    \begin{proof}
        Suppose there is a $(2n_1+1,...,2n_m+1)$ shaking of $L$ that bounds genus $0$ surfaces $\Sigma_1,...,\Sigma_m$ in the necessary way to make $L$ shake slice. In particular, $\Sigma_1$ bounds $2n_{11}+1$ copies of $L_1$, exactly $n_{11}+1$ of which have the same orientation as $L_1$, and $2n_{1j}$ copies of $L_j$, half of which have the same orientation as $L_j$, for each $1<j\leq m$. We may fuse the $n_{1k}$ pairs of oppositely orientated parallels of $L_k$ with bands, for all $k=1,...,m$. Note each band fusion of a pair of oppositely oriented parallels produces an unknot which we may push down and cap off in $B^4$. Hence, $L_1$ bounds a smooth surface of genus $n_{11}+\cdots+n_{1m}$. Similarly, each $L_j$ bounds a smooth surface of genus $n_{j1}+\cdots+n_{jm}$ and these surfaces are disjoint. Note the sum of the genera of these surfaces is $g_4(L)=\sum_{i=1}^m\sum_{j=1}^mn_{ij}=n_1+\cdots+n_m$.
    \end{proof}
    
    \begin{corollary}
    \label{cor:shake_disjoint}
    If $L$ is a shake slice link, then the pairwise linking numbers all vanish.
    \end{corollary}
    \begin{proof}
    Given any two components $L_i$ and $L_j$ of $L$, there exist disjoint smooth surfaces $\Sigma_i$ and $\Sigma_j$ in $B^4$ that bound $L_i$ and $L_j$, and therefore $lk(L_i,_j)=\Sigma_i\cdot\Sigma_j=0$.
    \end{proof}
    
    There is a much stronger obstruction from the higher order $\bar{\mu}$ invariants. The first non-vanishing Milnor $\bar{\mu}$ invariant is an invariant of shake concordance (see Theorem 5.2 in \cite{bosman20}). Therefore, since shake slice links are shake concordant to a trivial link, all the $\overline{\mu}$ invariants vanish when $L$ is shake slice, greatly generalizing Corollary \ref{cor:shake_disjoint}.

    Recall the Arf invariant, defined for knots, can be extended to proper links, that is, links $L$ such that \[\sum_{i\neq j}lk(L_i.L_j)\equiv0\pmod 2.\] Note that shake slice links and sublinks of shake slice links are proper as $lk(L_i,L_j)=0$ for $i\neq j$.  Suppose a planar surface embedded in $S^3\times[0,1]$ bounds $(L\times \{0\})\cup (K\times \{1\})$ for a proper link $L$ and some knot $K$. Then $Arf(K)$ depends only on $L$ so we may define $Arf(L):= Arf(K)$ for any such $K$ \cite[p. 45]{hillman12}.

\begin{theorem}
		If a link $L=L_1\sqcup L_2\sqcup ...\sqcup L_m$  is shake slice, then:
	\begin{itemize}
		\item $Arf(L)=0$,
		\item $Arf(L_i)=0$ for $i=1,...,m$, and
		\item $Arf(L_i \sqcup L_j)=0$ for $i\neq j$.
	\end{itemize}
	\end{theorem}
	\begin{proof}
	Since $L$ is shake slice, there exists $sh(L)$ a shaking of $L$ and a smoothly embedded genus surface in $S^3\times[0,1]$ that has boundary $L'\times \{0\}\cup T\times\{1\}$ where $T$ is a trivial link. Hence, by fusing the unknotted components of $T$, we have a planar surface cobounding $sh(L)$ and an unknot $U$, hence $Arf(sh(L))=Arf(U)$.
	We can fuse components of $L$ until we obtain some knot $K$ and hence there exists a planar surface bounding $L\times \{0\}\cup K\times\{1\}$. Therefore $Arf(L)=Arf(K)$. Note there also exists a planar surface bounding $sh(L)\times \{0\}\cup K\times\{1\}$ obtained by fusing pairs of parallel copies of each component of $L_i$ with opposite orientation to obtain $L$ then fusing the components of $L$ as before to obtain $K$. Hence,
	\[Arf(L)=Arf(K)=Arf(L')=Arf(U)=0.\]
	
	To see $Arf(L_i)=0$ for each component $L_i$ of $L$, observe that since $L$ is shake slice, there exists a sublink $L'_i$ of $sh(L)$ consisting of an odd number of parallel copies of $L_i$ and even number of parallel copies of each $L_j$ for $j\neq i$ such that a planar surface has boundary  $L'_i\times\{0\}\sqcup S\times\{1\}$ for some trivial link $S$. Then capping off all but one component of $S$, we have $Arf(L'_i)=Arf(U)$. But also notice that we can fuse pairs of parallel copies constituting $L'_i$ to obtain a planar surface cobounding $L'_i$ and $L_i$. Hence,
	\[Arf(L_i)=Arf(L'_i)=Arf(U)=0\]
	for all $i=1,...,m$.
	
	Finally, Beiss \cite{beiss90} has shown that for a two component link $L_1\sqcup L_2$ we have
	\[Arf(L_1\sqcup L_2)=Arf(L_1)+Arf(L_2)+\bar{\mu}_{L_1\sqcup L_2}(1122) \pmod{2}.\]
	Hence, since the Milnor invariants of $L$ all vanish, we have for any two-component sublink $L_i\sqcup L_j$ of $L$,
	\[Arf(L_i\sqcup L_j)=0,\]
	where $1\leq i<j\leq m$.
	\end{proof}

From the fact that the Milnor $\bar{\mu}$ invariants vanish for shake slice links, we have that shake slice links are null-homologous. Building off of work of Taniyama and Yasuhara in \cite{ty02}, Martin classified band-pass equivalence.

\begin{theorem}[Corollary 5.2 in \cite{martin13}]
For links $L=L_1\sqcup...\sqcup L_m$ and $L'=L_1'\sqcup...\sqcup L_m'$ with vanishing pairwise linking numbers, $L$ and $L'$ are band-pass equivalent if and only if:
\[Arf(L_i)=Arf(L_i')\]
\[\bar{\mu}_L(ijk)=\bar{\mu}_{L'}(ijk)\]
\[\bar{\mu}_L(iijj)\equiv\bar{\mu}_{L'}(iijj)\pmod{2}\]
for all $i,j,k\in\{1,...,m\}$.
\end{theorem}

Thus we have the following:
\begin{corollary}
If $L$ is shake slice, then $L$ is band-pass equivalent to the trivial link.
\end{corollary}

More broadly, we have that shake slice links behave like slice links on account of the following result.

\begin{proposition}
\label{homologically_slice}
Suppose the $m$-component link $L$ is shake slice. Then $L$ bounds $m$ disjoint disks in a homology 4-ball. That is, $L$ is homologically slice.
\end{proposition}
\begin{proof}
	Consider an $m$-component shake slice link $L$. Then $L$ is shake concordant to the $m$-component trivial link $T_m$. Note the zero surgery manifold $M_{T_m}$ is diffeomorphic to $\#_{i=1}^m S^1\times S^2$. Hence by Proposition \ref{homology_cobordant}	the zero surgery manifold $M_L$ is homology cobordant to $\#_{i=1}^m S^1\times S^2$. Let $W$ denote the 4-manifold of the homology cobordism. We modify $W$ to obtain a homology 4-ball. First, cap off $\#_{i=1}^m S^1\times S^2$ with $\natural m\, S^1\times D^3$ to obtain a 4-manifold which we denote $W'$. Notice $\partial W'=M_L$ and
	\[H_n(W')\cong
 	\begin{cases}
		\mathbb{Z} & n=0\\
		\mathbb{Z}^m & n=1\\
		0 & n\geq 2
	\end{cases}.
	\]
	Attach a 0-framed 2-handle to $W'$ along each of the $m$ meridians of $L$, denote the resulting 4-manifold $W''$. Note $\partial W''=0$. This kills the first homology group of $W'$, so that $W''$ is a homology ball. To see this consider the Mayer-Vietoris exact sequence:
\[\cdots\rightarrow H_1(\#_{i=1}^m S^1\times D^2)\xrightarrow{(i_*,j_*)} H_1(W')\oplus H_1(D^2\times D^2)\]\[\xrightarrow{k_*-l_*} H_1(W'')\xrightarrow{\partial_*} H_0(\#_{i=1}^m S^1\times D^2)\rightarrow\cdots.\]
We observe $H_1(D^2\times D^2)=0$ and $i_*$ is an isomorphism since $H_1(\#_{i=1}^m S^1\times D^2)$ is generated by the meridians of $L$. Hence, $k_*-l_*=0$. Moreover, $\partial_*$ is the zero map since $W''$ is connected and hence $H_1(W'')=0$.
The co-core of each 2-handle is a disk bounded by a component of $L_i$ in $W''$, $i=1,..,m$. Note these disks are disjoint.
\end{proof}

In particular, each component $L_i$ of $L$ is slice in a homology 4-ball. Cha, Livingston, and Ruberman have shown in \cite[Theorem 3]{clr08} that it then follows that $L_i$ is algebraically slice. Hence we obtain:

\begin{corollary}
If $L=L_1\sqcup...\sqcup L_m$ is shake slice, then each $L_i$ is algebraically slice. In particular, the signatures and Arf invariant vanish for each $L_i$, $i=1,...,m$.
\end{corollary}

%A link $L\subset S^3$ is called stably slice if there exists some $n\geq 0$ such that in $D^4\#nS^2\times S^2$ the components of $L$ bound disjoint, locally flat, nullhomologous disks.

%It follows from a result of Schneiderman [?] and Theorem ?? that a link is stably slice exactly when it is band pass equivalent to the trivial link. Hence, we have:
%\begin{propostion}
%Given a link $L\subset S^3$, if $L$ is shake slice, then it is stably slice.
%\end{propostion}

%Moreover, as it is known that stably slice links are exactly the set of 0-solvable links, as defined by Cochran-Orr-Teichner [?], we have that shake slice links are 0-solvable. 

\end{document}